\documentclass{amsart}
\usepackage{amssymb,amstext,amsmath,amscd,amsthm,amsfonts,enumerate,graphicx,latexsym,stmaryrd,multicol,bm}

\usepackage[usenames]{color}
\usepackage[all]{xy}
\newtheorem{thm}{Theorem}[section]
\newtheorem{lem}[thm]{Lemma}
\newtheorem{prop}[thm]{Proposition}
\newtheorem{cor}[thm]{Corollary}
\theoremstyle{definition}
\newtheorem{dfn}[thm]{Definition}
\newtheorem{ques}[thm]{Question}
\newtheorem{rmk}[thm]{Remark}

\newtheorem{eg}[thm]{Example}

\theoremstyle{remark}

\newtheorem*{ac}{Acknowledgments}
\newtheorem*{conv}{Convention}
\numberwithin{equation}{thm}
\def\add{\operatorname{add}}
\def\ann{\operatorname{ann}}

\def\c{\operatorname{C}}

\def\ca{\operatorname{ca}}
\def\cm{\operatorname{CM}}

\def\depth{\operatorname{depth}}
\def\dim{\operatorname{dim}}

\def\ext{\operatorname{ext}}
\def\Ext{\operatorname{Ext}}
\def\End{\operatorname{End}}
\def\lEnd{\underline{\operatorname{End}}}
\def\lmod{\underline{\operatorname{mod}}}
\def\fl{\operatorname{fl}}
\def\grade{\operatorname{grade}}
\def\gll{g\ell\ell}

\def\Hom{\operatorname{Hom}}

\def\inf{\operatorname{inf}}
\def\image{\operatorname{Im}}

\def\m{\mathfrak{m}}
\def\n{\mathfrak{n}}

\def\mod{\operatorname{mod}}

\def\nf{\mathrm{NF}}
\def\p{\mathfrak{p}}

\def\Rfd{\operatorname{Rfd}}
\def\sa{\operatorname{\underline{ann}}}

\def\spec{\operatorname{Spec}}
\def\sup{\operatorname{sup}}

\def\Tr{\operatorname{Tr}}
\def\V{\mathrm{V}}
\def\X{\mathcal{X}}
\def\xx{\bm{x}}
\def\Y{\mathcal{Y}}

\begin{document}
\allowdisplaybreaks
\title{Structure of modules stably annihilated by a fixed ideal}
\author{Yuki Mifune}
\address{Graduate School of Mathematics, Nagoya University, Furocho, Chikusaku, Nagoya 464-8602, Japan}
\email{yuki.mifune.c9@math.nagoya-u.ac.jp}
\thanks{2020 {\em Mathematics Subject Classification.} 13C14, 13C60}
\thanks{{\em Key words and phrases.} cohomology annihilator, finite representation type, generalized Loewy length, maximal Cohen--Macaulay module, minimal multiplicity, stable annihilator, syzygy category.}
\thanks{The author was partly supported by Grant-in-Aid for JSPS Fellows Grant Number 25KJ1386.}
\begin{abstract}
Let $R$ be a commutative noetherian ring, and denote by $\operatorname{mod} R$ the category of finitely generated $R$-modules. 
In this paper, for an ideal $I$ of $R$, we introduce the full subcategory $\operatorname{mod}_{I}(R)$ of $\operatorname{mod} R$ consisting of modules whose stable annihilators contain $I$, and we investigate its structure.
As an application, we explore the syzygy category of maximal Cohen--Macaulay $R$-modules, extending a theorem of Dey and Liu from the Gorenstein case to the Cohen--Macaulay case.
\end{abstract}
\maketitle
%\tableofcontents
%%%%%%%%%%%%%%%%%%%%
\section{Introduction}
Let $R$ be a commutative noetherian ring, and let $\mod R$ stand for the category of finitely generated $R$-modules.
We denote by $\sa_{R}(M)$ the {\em stable annihilator} of a finitely generated $R$-module $M$, which is by definition the set of elements $a \in R$ such that the multiplication map $M\xrightarrow{a}M$ factors through a projective $R$-module.
This notion has been studied in several works; see \cite{DKT2021,DeyLiu,E2020,Esenetepe,Kimura} for instance.

Given an ideal $I$ of $R$, we denote by $\mod_{I}(R)$ the full subcategory of $\mod R$ consisting of modules $M$ such that $\sa_{R}(M)$ contains $I$.
Following \cite{T2023}, we let $\c(R)$ stand for the full subcategory of $\mod R$ consisting of modules $M$ satisfying $\depth M_{\p}\geq\depth R_{\p}$ for every prime ideal $\p$ of $R$.
When $R$ is Cohen--Macaulay, $\c(R)$ coincides with the category of maximal Cohen--Macaulay $R$-modules, denoted by $\cm(R)$.
We put $\c_{I}(R)=\c(R)\cap\mod_{I}(R)$ and $\cm_{I}(R)=\cm(R)\cap\mod_{I}(R)$.
In this paper, we investigate the structural properties of the subcategories $\mod_{I}(R)$ and $\c_{I}(R)$.
Their basic closure properties are summarized in the following theorem.
%%%%%%%%%%%%%%%%%%%%
\begin{thm}[Lemma \ref{filt} and Propositions \ref{extofmodI} and \ref{extofCI}]\label{lem_int}
Let $R$ be a commutative noetherian ring and $I$ an ideal of $R$.
Then the following hold.
\begin{enumerate}[\rm(1)]
\item
The subcategory $\mod_{I}(R)$ contains $R$ and $R/I$, and is closed under finite direct sums, direct summands, transposes, $R$-duals, and syzygies.
Moreover, its extension closure coincides with the full subcategory of $\mod R$ consisting of modules that are locally free on $\spec R\setminus\V(I)$.
\item
The subcategory $\c_{I}(R)$ contains $R$, and is closed under finite direct sums, direct summands, and syzygies.
If $R$ has finite Krull dimension $d$, then $\c_{I}(R)$ is also closed under $d$-th syzygies of transposes and of $R$-duals.
Furthermore, if $(R,\m)$ is local and either $R$ is Gorenstein or $I$ is $\m$-primary, then its extension closure coincides with the full subcategory of $\c(R)$ consisting of modules that are locally free on $\spec R\setminus\V(I)$.
\end{enumerate}
\end{thm}
%%%%%%%%%%%%%%%%%%%%
When $R$ is a Cohen--Macaulay local ring and $I$ is the maximal ideal $\m$ of $R$, we obtain the following result, which is the main theorem of this paper.
Recall that a full subcategory of $\mod R$ is said to be {\em of finite type} if there exist only finitely many isomorphism classes of indecomposable modules in it.
\begin{thm}[Theorem \ref{mainthm}]\label{mainthm_int}
Let $(R,\m,k)$ be a singular Cohen--Macaulay local ring of dimension $d$.
Then one has
\[
\Omega(\cm_{\m}(R))\subseteq\left|\Omega^{d}k\oplus R\right|_{\gll(R)-1}.
\]
Moreover, if $\gll(R)=2$ or $R$ has minimal multiplicity, then the equality 
\[
\Omega(\cm_{\m}(R))=\add(\Omega^{d}k\oplus R)
\]
holds.
Hence, the category $\Omega(\cm_{\m}(R))$ is of finite type in this case.
\end{thm}
%%%%%%%%%%%%%%%%%%%%
Let us explain the notation used in the above theorem.
For a full subcategory $\X$ of $\mod R$, we denote by $\Omega\X$ the (first) syzygy category of $\X$.
For a finitely generated $R$-module $X$ and an integer $n \ge 0$, $\left|X\right|_{n}$ consists of the modules built from $X$ by taking $(n-1)$ extensions (for the precise definition, see Definition \ref{def_size}).
We denote by $\add X$ the additive closure of $X$ in $\mod R$.
In addition, we denote by $\gll(R)$ the {\em generalized Loewy length} of $R$, defined as the infimum of the Loewy lengths $\ell\ell(R/(\bm{x}))$ taken over all systems of parameters $\bm{x}$ of $R$.
%%%%%%%%%%%%%%%%%%%%

As an application of Theorem \ref{mainthm_int}, we obtain the following result, which provides a sufficient condition for the syzygy category of maximal Cohen--Macaulay $R$-modules to be of finite type.
\begin{cor}[Corollary \ref{cor_fin_ocm}]\label{cor_int}
Let $(R,\m,k)$ be a singular Cohen--Macaulay local ring of dimension $d$.
Assume that the cohomology annihilator $\ca^{d+1}(R)$ contains $\m$, and that either $\gll(R)=2$ or $R$ has minimal multiplicity.
Then the category $\operatorname{\Omega CM}(R)$ is of finite type.
\end{cor}
%%%%%%%%%%%%%%%%%%%%
Here, for $n\ge0$, the ($n$-th) {\em cohomology annihilator} of $R$, denoted $\ca^{n}(R)$, is defined as the set of elements of $R$ that annihilate the $R$-module $\Ext_{R}^{n}(M, N)$ for all finitely generated $R$-modules $M,N$.
This concept was introduced in \cite{Iyengar Takahashi 2016}.
By restricting Corollary \ref{cor_int} to the case where $R$ is Gorenstein, we can recover the recent theorem of Dey and Liu \cite[Theorem 1.2(1)]{DeyLiu}:
\begin{cor}[Dey--Liu]
Let $(R,\m)$ be a Gorenstein local ring whose cohomology annihilator coincides with $\m$.
If $R$ has minimal multiplicity, then $\cm(R)$ is of finite type.
\end{cor}
The organization of this paper is as follows.
In Section 2, we first recall some basic definitions and known results. We then investigate the structural properties of $\mod_{I}(R)$ and establish several results, which also serve as preliminaries for the proof of the main theorem.
In Section 3, in the case where $R$ is Cohen--Macaulay, we prove a theorem concerning $\Omega(\cm_{\m}(R))$ and explore its applications.
\begin{conv}
Throughout this paper, all rings are assumed to be commutative noetherian.
All subcategories are assumed to be full and closed under isomorphism.
\end{conv}
%%%%%%%%%%%%%%%%%%%%
\section{Structural properties of $\operatorname{mod}_{I}(R)$ and $\operatorname{C}_{I}(R)$}
In this section, we first recall some basic notions concerning the stable annihilators of modules. 
We then investigate the properties of $\mod_{I}(R)$. 
In addition, mainly in the case where $R$ is a local ring, we study $\c_{I}(R)$, which is defined as the intersection of $\c(R)$ and $\mod_{I}(R)$, together with its syzygy category.

We start by recalling the notions of syzygies and additive closures of modules.
%%%%%%%%%%%%%%%%%%%%
\begin{dfn}\label{def_syz}
Let $R$ be a commutative noetherian ring.
\begin{enumerate}[\rm(1)]
\item
Let $M$ be a finitely generated $R$-module.
Consider a projective resolution of $M$ in $\mod R$:
\[
\cdots \xrightarrow{\delta_{n+1}} P_n \xrightarrow{\delta_n} P_{n-1} \xrightarrow{\delta_{n-1}}\cdots \xrightarrow{\delta_1}P_0 \xrightarrow{\delta_0}M \to 0.
\]
For each $n\geq0$, we define the {\em $n$-th syzygy of $M$} by $\Omega^{n}M=\image(\delta_{n})$.
Note that $\Omega^{n} M$ is uniquely defined up to projective summands.
When $R$ is a local ring, we take a {\em minimal free resolution} to define $\Omega^{n} M$.
In this case, $\Omega^{n} M$ is uniquely determined up to isomorphism.
\item
Let $\X$ be a subcategory of $\mod R$.
For a positive integer $n$, we denote by $\Omega^{n}\X$ the {\em $n$-th syzygy category} of $\X$, that is, the subcategory of $\mod R$ consisting of all $R$-modules $M$ such that there exists an exact sequence of the form $0 \to M \to F_{n-1} \to \cdots \to F_0 \to X \to 0 $, where $X \in \X$ and each $F_i$ is a finitely generated projective $R$-module. 
We set $\Omega^{0}\X=\X$.
\item
Let $M$ be a finitely generated $R$-module, and take a projective presentation $P_{1}\xrightarrow{\delta_1}P_0 \xrightarrow{\delta_0}M \to 0$ of $M$.
We define the {\em (Auslander) transpose} of $M$, denoted by $\Tr M$, as the cokernel of the homomorphism $\Hom_{R}(\delta_{1},R)$.
If $R$ is a local ring, we take the above presentation to be induced by a minimal free resolution of $M$.
\item
For a subcategory $\X$ of $\mod R$, we denote by $\add\X$ the {\em additive closure} of $\X$, that is, the subcategory of $\mod R$ consisting of direct summands of finite direct sums of modules in $\X$.
Note that $\add\X$ is the smallest subcategory of $\mod R$ that contains $\X$ and is closed under finite direct sums and direct summands.
\item
A subcategory $\X$ of $\mod R$ is {\em resolving} if $\X$ contains $\add_{R}R$ and is closed under direct summands, extensions, and kernels of epimorphisms in $\mod R$ (see \cite{APST, AB, AR, KS, stcm, arg} for instance).
\end{enumerate}
\end{dfn}
%%%%%%%%%%%%%%%%%%%%
\begin{rmk}
Let $R$ be a commutative noetherian ring.
\begin{enumerate}[\rm(1)]
\item
Let $\X$ be a subcategory of $\mod R$.
If $R$ is a local ring, then
\begin{equation}
\Omega^{n} \X = \nonumber
\begin{cases}
\X & \text{if $n=0$,} \\
\{\Omega^{n} X \oplus R^{\oplus m} \mid X\in \X, m\geq 0 \}& \text{if $n>0$.} 
\end{cases}
\end{equation}
\item
The subcategory $\add R$ coincides with the category of finitely generated projective $R$-modules.
Moreover, $\add R$ is a resolving subcategory.
\end{enumerate}
\end{rmk}
%%%%%%%%%%%%%%%%%%%%
Next, we define the category $\c(R)$, which was introduced in \cite{T2023}, and each module in $\c(R)$ can be regarded as a generalization of a maximal Cohen--Macaulay $R$-module.
\begin{dfn}\label{defofCR}
Let $R$ be a commutative noetherian ring.
\begin{enumerate}[\rm(1)]
\item 
We denote by $\c(R)$ the subcategory of $\mod R$ consisting of $R$-modules $M$ satisfying the inequality $\depth M_{\p}\geq \depth R_{\p}$ for all $\p\in\spec R$.
If $R$ is Cohen--Macaulay, then $\c(R)$ coincides with the category of maximal Cohen--Macaulay $R$-modules, which we denote by $\cm(R)$.
\item
For a finitely generated $R$-module $M$, we define the {\em (large) restricted flat dimension} of $M$, 
denoted by $\Rfd(M)$, as follows.
\begin{center}
$\Rfd(M) 
= \sup \{\depth R_{\p}-\depth M_{\p} \mid
\p \in \spec R \}$. 
\end{center}
By virtue of \cite[Theorem 1.1]{Avramov Iyengar Lipman} and \cite[Proposition 2.2, Theorem 2.4]{Christensen Foxby Frankild}, we have 
$\Rfd(M) \in \mathbb{Z}_{\geq 0} \cup \{-\infty \}$. 
Clearly, one has $\Rfd(M)\leq\dim R$.
\end{enumerate}
\end{dfn}
%%%%%%%%%%%%%%%%%%%%
\begin{rmk}\label{rmk_defofCR}
Let $R$ be a commutative noetherian ring.
\begin{enumerate}[\rm(1)]
\item
By \cite[Proposition 1.2.9]{BH}, the category $\c(R)$ is resolving.
\item
Let $M$ be an $R$-module in $\c(R)$.
Then every weak $R$-sequence is also a weak $M$-sequence.
\item
For any finitely generated $R$-module $M$, and for all integers $i\geq\Rfd(M)$, one has $\Omega^{i}M\in\c(R)$.
If $R$ has finite Krull dimension $d$, then $\Omega^{d}(\mod R)\subseteq \c(R)$.
\end{enumerate}
\end{rmk}
%%%%%%%%%%%%%%%%%%%%
Next, we define the stable annihilators and the nonfree loci of modules.
\begin{dfn}\label{nfloc_stann}
Let $R$ be a commutative noetherian ring, and $M$ a finitely generated $R$-module.
\begin{enumerate}[\rm(1)]
\item
We denote by $\lmod(R)$ the stable category of $\mod R$ (cf. \cite{AB}). We write $\lEnd_{R}(M)=\End_{\lmod(R)}(M)$; that is, $\lEnd_{R}(M)$ is the quotient of $\End_{R}(M)$ by the homomorphisms $M\to M$ factoring through a finitely generated projective $R$-module.
The {\em stable annihilator} of $M$ refers to the annihilator of $\lEnd_{R}(M)$ as an $R$-module. We denote it by $\sa_{R}(M)$.
For a subcategory $\X$ of $\mod R$, we set $\sa_{R}(\X)=\bigcap_{X\in\X}\sa_{R}(X)$.
\item
We denote by $\nf(M)$ the {\em nonfree locus} of $M$, that is, the set of prime ideals $\p$ such that $M_{\p}$ is not a free $R_{\p}$-module.
\item
Let $\Phi$ be a subset of $\spec R$. 
We denote by $\mod_{\Phi}(R)$ the subcategory of $\mod R$ consisting of $R$-modules $M$ such that $\nf(M)\subseteq\Phi$.
We set $\c_{\Phi}(R)=\c(R)\cap\mod_{\Phi}(R)$.
If $R$ is Cohen--Macaulay, we denote $\c_{\Phi}(R)$ by $\cm_{\Phi}(R)$.
\end{enumerate}
\end{dfn}
%%%%%%%%%%%%%%%%%%%%
\begin{rmk}\label{rmk_nfloc_stann}
Let $R$ be a commutative noetherian ring.
\begin{enumerate}[\rm(1)]
\item
For finitely generated $R$-modules $M$ and $N$, the following statements hold (see \cite{AB}, \cite[Lemma 2.2]{DKT2021}, \cite[Lemma 2.3]{Esenetepe}, \cite[Lemma 2.14]{Iyengar Takahashi 2016}, and \cite[Proposition 2.10]{T2009} for instance).
\begin{enumerate}[\rm(a)]
\item
The equalities $\sa_{R}(M)=\ann_{R}\Ext_{R}^{1}(M,\Omega M)=\ann_{R}\Ext_{R}^{>0}(M,\mod R)$ hold.
\item
One has $\nf(M)=\V(\sa_{R}(M))$.
\item
The inclusion $\nf(M)\subseteq\V(I)$ holds if and only if there exists an integer $n\geq0$ such that $I^{n}\subseteq \sa_{R}(M)$.
\item
One has $\sa_{R}(M)=R$ if and only if $M$ is a projective $R$-module.
\item
The equality $\sa_{R}(M\oplus N)=\sa_{R}(M)\cap\sa_{R}(N)$ holds.
\item
The equality $\sa_{R}(M)=\sa_{R}(\Tr_{R}(M))$ holds.
\item
The inclusion $\sa_{R}(M)\subseteq\sa_{R}(\Hom_{R}(M,R))$ holds.
\item
The inclusion $\sa_{R}(M)\subseteq\sa_{R}(\Omega M)$ holds.
\end{enumerate}
\item
For a subset $\Phi$ of $\spec R$, the category $\mod_{\Phi}(R)$ is a resolving subcategory.
Hence, the category $\c_{\Phi}(R)$ is also a resolving subcategory.
\item
Assume that $(R,\m)$ is a local ring. 
Recall that $R$ has an {\em isolated singularity} if the localization $R_{\p}$ is a regular local ring for every prime ideal $\p\neq\m$.
It follows from \cite[Theorems 1.3.3 and 2.2.7]{BH} that the equality $\c_{\V(\m)}(R)=\c(R)$ holds if and only if $R$ has an isolated singularity.
If $t=\depth R,$ then $\Omega^{t}(\mod_{\V(\m)}(R)) \subseteq \c_{\V(\m)}(R)$.
\item
Let $(R,\m,k)$ be a Cohen--Macaulay local ring with a canonical module $\omega$.
By virtue of \cite[Theorem 2.3]{DKT2021}, the ring $R$ is nearly Gorenstein if and only if $\omega\in\mod_{\m}(R)$.
\end{enumerate}
\end{rmk}
%%%%%%%%%%%%%%%%%%%%
Next, using the stable annihilators of modules and an ideal $I$ of the ring $R$, we define the subcategory $\mod_{I}(R)$, which is the main focus of this paper.
\begin{dfn}\label{defofmodIR}
Let $R$ be a commutative noetherian ring and $I$ an ideal of $R$.
We denote by $\mod_{I}(R)$ the subcategory of $\mod R$ consisting of modules $M$ such that $I\subseteq\sa_{R}(M)$.
We define $\c_{I}(R)=\c(R)\cap\mod_{I}(R)$.
When $R$ is Cohen--Macaulay, we set $\cm_{I}(R)=\cm(R)\cap\mod_{I}(R)$.
\end{dfn}
%%%%%%%%%%%%%%%%%%%%
We first state some basic properties of $\mod_{I}(R)$ and $\c_{I}(R)$, such as their closure properties.
\begin{lem}\label{filt}
Let $R$ be a commutative noetherian ring and $I$ an ideal of $R$.
\begin{enumerate}[\rm(1)]
\item\label{filt_incl}
If $J$ is an ideal of $R$ such that $I\subseteq J$, then one has $\mod_{J}(R)\subseteq\mod_{I}(R)$.
\item\label{filt_filt}
There exists an ascending chain of subcategories of $\mod R$ as follows:
\begin{equation}\label{modIfilt}
\add R=\mod_{I^{0}}(R)\subseteq\mod_{I}(R)\subseteq\mod_{I^{2}}(R)\subseteq\cdots\subseteq\displaystyle\bigcup_{n\geq0}\mod_{I^{n}}(R)=\mod_{\V(I)}(R).
\end{equation}
\item\label{filt_closure}
The subcategory $\mod_{I}(R)$ is closed under finite direct sums, direct summands, transposes, $R$-duals, and syzygies.
Moreover, $\mod_{I}(R)$ contains modules $R/I$ and $I$.
\item\label{filt_filt_c}
There exists an ascending chain of subcategories of $\c(R)$ as follows:
\begin{equation}\label{CIfilt}
\add R=\c_{I^{0}}(R)\subseteq\c_{I}(R)\subseteq\c_{I^{2}}(R)\subseteq\cdots\subseteq\displaystyle\bigcup_{n\geq0}\c_{I^{n}}(R)=\c_{\V(I)}(R).
\end{equation}
\item
\begin{enumerate}[\rm(a)]
\item
The subcategory $\c_{I}(R)$ is closed under finite direct sums, direct summands, and syzygies.
\item 
If $R$ has finite Krull dimension, then the subcategory $\c_{I}(R)$ is closed under $\Omega^{i}\Tr(-)$ and $\Omega^{i}\Hom_{R}(-,R)$ for all $i\geq\dim R$.
\end{enumerate}
\end{enumerate}
\end{lem}
\begin{proof}
(\ref{filt_incl}) Clear.

(\ref{filt_filt}) The equalities and inclusions follow from Remark \ref{rmk_nfloc_stann} and (\ref{filt_incl}), respectively. 

(\ref{filt_closure}) The closure properties follow from Remark \ref{rmk_nfloc_stann}.
Since $\sa_{R}(R/I)\supseteq\ann_{R}\End_{R}(R/I)=I$, we have $R/I\in\mod_{I}(R)$.
Moreover, one has $I=\Omega(R/I)\in\mod_{I}(R)$.

(\ref{filt_filt_c}) By taking the intersection with $\c(R)$ in (\ref{modIfilt}), we obtain the conclusion.

(5)(a) Since both subcategories $\c(R)$ and $\mod_{I}(R)$ are closed under finite direct sums, direct summands, and syzygies, it follows that $\c_{I}(R)$ is also closed under these operations.
The assertion in (b) follows from Remark \ref{rmk_defofCR} and (\ref{filt_closure}).
\end{proof}
%%%%%%%%%%%%%%%%%%%%
\begin{rmk}
The filtrations of subcategories in (\ref{modIfilt}) and (\ref{CIfilt}) do not stabilize in general.
We now give examples where the equalities fail.
\begin{enumerate}[\rm(1)]
\item
Let $R$ be a noetherian local ring and $I$ a proper ideal of $R$ with $\grade I>0$.
Then for each integer $n>0$, we have $\sa_{R}(R/I^{n})=\ann_{R}I^{n}+I^{n}=I^{n}$, where the first equality follows from \cite[Lemma 3.6]{DeyLiu}.
Hence by Nakayama's lemma, we obtain $R/I^{n}\in\mod_{I^{n}}(R)\setminus\mod_{I^{n-1}}(R)$ for all $n>0$.
\item
Consider $R=k[\![x,y]\!]/(x^{2})$ and for each $n>0$, let $M_{n}$ be the cokernel of the endomorphism of $R^{\oplus 2}$ defined by the matrix 
$\left(\begin{smallmatrix}\overline{x}&\overline{y}^{n}\\0&-\overline{x}\end{smallmatrix}\right)$.
Then one has $\sa_{R}(M_{n})=(\overline{x},\overline{y}^{n})$ by \cite[Proposition 3.3(1)]{Kimura}.
Hence, for $I=(\overline{y})$, we obtain $M_{n}\in\cm_{I^{n}}(R)\setminus\cm_{I^{n-1}}(R)$ for all $n>0$.
\end{enumerate}
\end{rmk}
%%%%%%%%%%%%%%%%%%%%
Next, for a given subcategory $\X$ of $\mod R$, we define the subcategory obtained from modules in $\X$ by taking a fixed number of extensions. This concept was introduced in \cite{DT_rad} and \cite{DT_dim}.
\begin{dfn}\label{def_size}
\begin{enumerate}[\rm(1)]
\item
For subcategories $\X$, $\Y$ of $\mod R$, we define $\X\ast\Y$ as the subcategory of $\mod R$ consisting of objects $Z$ such that there exists an exact sequence $0\to X\to Z\to Y\to0$ with $X\in\X$, and $Y\in\Y$.
\item
Let $\X$ be a subcategory of $\mod R$, and $r$ a nonnegative integer.
We inductively define ${\left|\X\right|}_{r}^{R}$ as follows.
\begin{equation}
{\left| \X \right|}_{r}^{R} = \nonumber
\begin{cases}
0 & \text{if $r=0$,} \\
\add\X & \text{if $r=1$,} \\
\add\left({{\left| \X \right|}_{r-1}^{R}} \ast {\left|\X \right|}^{R}_{1}\right) & \text{if $r>1$.}
\end{cases}
\end{equation}
\item
For a subcategory $\X$ of $\mod R$, we denote by $\ext\X$ the smallest subcategory of $\mod R$ that contains $\X$ and is closed under finite direct sums, direct summands, and extensions.
Note that there is a filtration of subcategories of $\mod R$ given by
\[
0\subseteq\add\X\subseteq{\left| \X \right|}_{2}^{R} \subseteq\cdots\subseteq\bigcup_{i\geq0}{\left| \X \right|}_{i}^{R}=\ext\X.
\]
\end{enumerate}
\end{dfn}
%%%%%%%%%%%%%%%%%%%%
Next, we investigate the relationship between $\mod_{I}(R)$ and $\mod_{\V(I)}(R)$. 
In general, $\mod_{I}(R)$ is not closed under extensions and hence is not resolving. 
However, its extension closure coincides with $\mod_{\V(I)}(R)$.
\begin{prop}\label{extofmodI}
Let $R$ be a commutative noetherian ring and $I$ an ideal of $R$.
Then one has 
\[
\ext(\mod_{I}(R))=\mod_{\V(I)}(R).
\] 
In particular, the category $\mod_{I}(R)$ is closed under extensions if and only if the equality $\mod_{I}(R)=\mod_{\V(I)}(R)$ holds.
\end{prop}
\begin{proof}
Since $\mod_{\V(I)}(R)$ is resolving, the inclusion $\ext(\mod_{I}(R))\subseteq\mod_{\V(I)}(R)$ holds.
By Lemma \ref{filt}, we have 
\[
\mod_{I}(R)\supseteq\{\Omega^{i}(R/\p)\mid 0\leq i\leq \mu(I),\ \p\in \V(I)\}\cup\{R\},
\]
where $\mu(I)$ denotes the minimal number of generators of $I$.
Hence, it follows that 
\[
\ext(\mod_{I}(R))\supseteq\ext\left(\{\Omega^{i}(R/\p)\mid 0\leq i\leq \mu(I),\ \p\in \V(I)\}\cup\{R\}\right)=\mod_{\V(I)}(R),
\]
where the equality follows from \cite[Theorem 3.1]{BHST}.
\end{proof}
%%%%%%%%%%%%%%%%%%%%
Here, we consider an analogue of Proposition \ref{extofmodI} for $\c_{I}(R)$. 
We will see that this holds under certain assumptions on the ring $R$ or on the ideal $I$.
\begin{prop}\label{extofCI}
Let $(R,\m,k)$ be a noetherian local ring and $I$ an ideal of $R$.
Assume that one of the following conditions holds.
\begin{enumerate}[\rm(1)]
\item
The ring $R$ is Gorenstein.
\item
The ideal $I$ is $\m$-primary.
\end{enumerate}
Then one has 
$\ext(\c_{I}(R))=\c_{\V(I)}(R)$.
In particular, in this case, the category $\c_{I}(R)$ is closed under extensions if and only if the equality $\c_{I}(R)=\c_{\V(I)}(R)$ holds.
\end{prop}
\begin{proof}
The inclusion $\ext(\c_{I}(R))\subseteq\c_{\V(I)}(R)$ is immediate.
We begin by assuming (1).
We claim that the following equality holds:
\[
\cm_{\V(I)}(R)=\ext\left(\{\Omega^{-d}\Omega^{d+i}(R/\p)\mid 0\leq i\leq \mu(I),\ \p\in \V(I)\}\cup\{R\}\right),
\]
where $d=\dim R$.
Note that the syzygy functor on the stable category of $\cm(R)$ is an equivalence, with quasi-inverse given by the cosyzygy functor $\Omega^{-1}(-)$.
This implies that for any maximal Cohen--Macaulay $R$-module $M$, one has 
\[
\sa_{R}(M)=\sa_{R}(\Omega^{-1}M).
\]
In particular, $\cm_{I}(R)$ is closed under cosyzygies.
Combining this with the fact that $\cm_{\V(I)}(R)$ is resolving, we obtain the inclusion 
\[
\cm_{\V(I)}(R)\supseteq\ext\left(\{\Omega^{-d}\Omega^{d+i}(R/\p)\mid 0\leq i\leq \mu(I),\ \p\in \V(I)\}\cup\{R\}\right).
\]
Conversely, let $M$ be an $R$-module in $\cm_{\V(I)}(R)$.
Set $\X=\{\Omega^{i}(R/\p)\mid 0\leq i\leq \mu(I),\ \p\in \V(I)\}\cup\{R\}$.
By \cite[Theorem 3.1]{BHST}, we have $M\in\ext\X$.
Since $M$ is maximal Cohen--Macaulay, it is stably isomorphic to $\Omega^{-d}\Omega^{d}M$. 
Hence,
\[
M\in\add\left(\Omega^{-d}\Omega^{d}\left(\ext\X\right)\cup\{R\}\right)\subseteq\ext\left(\Omega^{-d}\Omega^{d}\X\cup\{R\}\right).
\]
This shows that 
\[
M\in\ext\left(\{\Omega^{-d}\Omega^{d+i}(R/\p)\mid 0\leq i\leq \mu(I),\ \p\in \V(I)\}\cup\{R\}\right).
\]
Hence, the assertion follows from the equality and a similar argument as in the proof of Proposition \ref{extofmodI}.

Next, we consider the case where (2) holds.
Let $M$ be an $R$-module in $\c_{\V(I)}(R)$.
Set $J=\sa_{R}(M)$.
Since $\grade I=\depth R=t$ and $I^{n}\subseteq J$ for some $n\ge1$, there exists an $R$-sequence $x_{1},\ldots,x_{t}$ contained in $J$.
We may choose additional elements $x_{t+1},\ldots,x_{m}$ so that $\bm{x}=x_{1},\ldots,x_{m}$ generates $J$.
By \cite[Corollary 3.2]{kos}, there exists an exact sequence 
\begin{equation}\label{ex_kos}
0\to \mathrm{H}_{i}(\bm{x},M)\to E_{i}\to \Omega E_{i-1}\to 0
\end{equation}
for each $1\leq i \leq m$, where $E_{0}=\mathrm{H}_{0}(\bm{x},M)$, and $M$ is a direct summand of $E_{m}$.
Since $M\in\c(R)$, the sequence $x_{1},\ldots,x_{t}$ is also an $M$-sequence.
Thus, $\grade(J,M)\geq t$.
Then the grade sensitivity (\cite[Theorem 1.6.17]{BH}) implies that $\mathrm{H}_{m}(\bm{x},M)=\cdots=\mathrm{H}_{m-t+1}(\bm{x},M)=0$.
Combining this with (\ref{ex_kos}), we see that $E_{m} \cong \Omega^{t}E_{m-t}$.
On the other hand, the module $E_{m-t}$ belongs to $\left|\bigoplus_{i=0}^{m-t}\Omega^{i}\mathrm{H}_{m-t-i}(\bm{x},M)\right|_{m-t+1}^{R}$ by (\ref{ex_kos}).
Note that $\Omega^{i}(\mod R/J)\subseteq\left|\Omega^{i}(\mod R/I)\right|_{n}^{R}\subseteq\left|\c_{I}(R)\right|_{n}^{R}$ for all $i\ge t$.
Therefore,
\[
\Omega^{t}E_{m-t}\in\left|\bigoplus_{i=t}^{m}\Omega^{i}\mathrm{H}_{m-i}(\bm{x},M)\right|_{m-t+1}^{R}\subseteq\ext(\c_{I}(R)).
\]
This shows that $M\in\ext(\c_{I}(R))$.
\end{proof}
%%%%%%%%%%%%%%%%%%%%
\begin{rmk}\label{BHST4.1}
Combining Remark \ref{rmk_defofCR} with \cite[Theorem 4.1]{BHST}, one has
\[
\c_{\V(\m)}(R)=\left|\bigcup_{i=t}^{d}\Omega^{i}(\fl(R))\right|_{d-t+1}^{R}=\ext_{R}\left(\bigoplus_{i=t}^{d}\Omega^{i}k\oplus R\right),
\]
where $\fl(R)$ denotes the subcategory of $\mod R$ consisting of modules of finite length.
Hence, using this equality and an argument similar to the proof of Proposition \ref{extofmodI}, we obtain the result of Proposition \ref{extofCI}(2) in the case where $I=\m$.
\end{rmk}
%%%%%%%%%%%%%%%%%%%%
Based on Propositions \ref{extofmodI} and \ref{extofCI}, the following question arises naturally.
\begin{ques}
Let $R$ be a commutative noetherian ring, and $I$ an ideal of $R$. 
Does the following equality hold?
\[
\ext(\c_{I}(R))=\c_{\V(I)}(R).
\]
\end{ques}
%%%%%%%%%%%%%%%%%%%%
The notion of the cohomology annihilator of $R$ was introduced in \cite{Iyengar Takahashi 2016}. 
Here, we study the relationship between $\c_{I}(R)$ and the cohomology annihilator of $R$.
\begin{dfn}
Let $R$ be a commutative noetherian ring.
For a nonnegative integer $n$, we define the {\em $n$-th cohomology annihilator} of $R$, denoted by $\ca^{n}(R)$, to be
\begin{center}
$\ca^{n}(R)=
\displaystyle\bigcap_{m\ge n}
\left(
\displaystyle\bigcap_{M, N \in \mod R}
\ann_{R}\Ext_{R} ^{m}(M, N)
\right)$.
\end{center}
In addition, we define the {\em cohomology annihilator} of $R$ by
\begin{center}
$\ca(R)=
\displaystyle\bigcup_{n\ge 0} \ca^{n}(R)$.
\end{center}
Note that the ascending chain of cohomology annihilators
$0=\ca^{0}(R) \subseteq \ca^{1}(R) \subseteq \ca^{2}(R) \subseteq \cdots$ stabilizes.
\end{dfn}
%%%%%%%%%%%%%%%%%%%%
\begin{rmk}
Let $R$ be a commutative noetherian ring.
\begin{enumerate}[\rm(1)]
\item
The equality $\ca(R)=R$ holds if and only if $R$ is a regular ring with finite Krull dimension.
\item
For an integer $n\geq0$, one has 
\[
\ca^{n+1}(R)=\bigcap_{M, N \in \mod R}\ann_{R}\Ext_{R} ^{n+1}(M, N)=\sa_{R}(\Omega^{n}(\mod R)).
\]
\end{enumerate}
\end{rmk}
%%%%%%%%%%%%%%%%%%%%
The following lemma gives a necessary condition  for $\c_{I}(R)$ to be extension closed. 
\begin{lem}\label{neces}
Let $(R,\m,k)$ be a $d$-dimensional noetherian local ring, and $I$ an $\m$-primary ideal.
If the category $\c_{I}(R)$ is extension closed, then the ring $R$ has an isolated singularity.
\end{lem}
\begin{proof}
By Proposition \ref{extofCI}, one has $\c_{I}(R)=\c_{\V(I)}(R)=\c_{\V(\m)}(R)$.
Let $\p$ be a prime ideal of $R$ with $\p\neq\m$.
By Remark \ref{rmk_defofCR}, we have $\Omega^{d}(R/\p)\in\c(R)$.
Hence,
\[
\sa_{R}(\Omega^{d}(R/\p))\supseteq\sa_{R}(\c(R)).
\]
It follows from \cite[Proposition 2.4(2)]{Kimura} that
\[
\sqrt{\sa_{R}(\c(R))}=\sqrt{\sa_{R}(\c_{\V(\m)}(R))}.
\]
Combining this with the assumption, we obtain 
\[
\sqrt{\sa_{R}(\c(R))}=\sqrt{\sa_{R}(\c_{I}(R))}\supseteq\sqrt{I}=\m.
\]
Thus,
\[
\nf(\Omega^{d}(R/\p))=\V(\sa_{R}(\Omega^{d}(R/\p)))\subseteq\V(\sa_{R}(\c(R)))\subseteq\V(\m).
\]
This implies that $\p$ does not belong to $\nf(\Omega^{d}(R/\p))$, and hence $R_{\p}$ is regular.
\end{proof}
%%%%%%%%%%%%%%%%%%%%
The condition $\cm_{\m}(R)=\cm_{\V(\m)}(R)$ can be characterized in terms of the cohomology annihilator.
\begin{prop}
Let $(R,\m,k)$ be a Cohen--Macaulay local ring of dimension $d$.
Then the following are equivalent.
\begin{enumerate}[\rm(1)]
\item 
The equality $\cm_{\m}(R)=\cm(R)$ holds.
\item
The equality $\cm_{\m}(R)=\cm_{\V(\m)}(R)$ holds.
\item 
One has $\m\subseteq\ca^{d+1}(R)$.
\end{enumerate}
\end{prop}
\begin{proof}
By Lemma \ref{neces}, the condition (2) implies that $R$ has an isolated singularity.
Combining this with Remark \ref{rmk_nfloc_stann} and \cite[Lemma 2.10]{Iyengar Takahashi 2016}, it suffices to prove the equivalence between conditions (1) and (3) under the assumption that $R$ has an isolated singularity.
Since $\Omega^{d}(\mod R)=\cm(R)$ by \cite[Theorem 3.8]{EG}, we have $\sa_{R}(\Omega^{d}(\mod R))=\sa_{R}(\cm(R))$.
Hence, condition (1) holds if and only if one has $\m\subseteq\sa_{R}(\cm(R))$, if and only if one has $\m\subseteq\sa_{R}(\Omega^{d}(\mod R))$, that is, if and only if condition (3) holds.
\end{proof}
%%%%%%%%%%%%%%%%%%%%
The following lemma describes a certain closure property of $\mod_{I}(R)$ with respect to short exact sequences.
\begin{lem}\label{lprojext}
Let $R$ be a commutative noetherian ring and $I$ an ideal of $R$.
Assume that there exists an exact sequence
\[
0\to P\to A\to B\to0
\]
in $\mod R$ with $P\in\add_{R}R$.
Then we have $\sa_{R}(B)\subseteq\sa_{R}(A)$.
In particular, if $B\in\mod_{I}(R)$, then $A\in\mod_{I}(R)$ as well. In other words, we have $\add R\ast\mod_{I}(R)\subseteq\mod_{I}(R)$.
\end{lem}
\begin{proof}
Let $M$ be a finitely generated $R$-module.
Applying the functor $\Hom_{R}(-,M)$ to the exact sequence, we obtain the epimorphism $\Ext_{R}^{i}(B,M)\to\Ext_{R}^{i}(A,M)$ for each $i>0$.
This implies that $\sa_{R}(B)\subseteq\sa_{R}(A)$.
The latter assertion follows from the former.
\end{proof}
%%%%%%%%%%%%%%%%%%%%
The following two lemmas will be used in the proof of the main theorem in the next section. The first lemma states that the first syzygy category of $\c_{\m}(R)$ is closed under direct summands.
\begin{lem}\label{omega_smd}
Let $R$ be a commutative noetherian ring and $I$ an ideal of $R$.
Let $\X$ be a subcategory of $\mod R$ that contains $\add R$ and is closed under direct summands and extensions.
Then the subcategory $\Omega(\X\cap\mod_{I}(R))$ is closed under direct summands.
In particular, if $(R,\m,k)$ is a noetherian local ring, then $\Omega(\c_{\m}(R))$ is closed under direct summands.
\end{lem}
\begin{proof}
Let $A$ be a module in $\Omega(\X\cap\mod_{I}(R))$, and assume that $A\cong A_{1}\oplus A_{2}$.
Then there exists an exact sequence
\[
0\to A_{1}\oplus A_{2}\to P\to B\to0
\]
in $\mod R$, where $P\in\add R$ and $B\in\X\cap\mod_{I}(R)$.
By \cite[Lemma 3.1]{T2006} we obtain the following exact sequences:
\begin{equation}\label{omega_smd_f1}
0\to A_{i}\to P\to C_{i}\to0 \quad (i=1,2),
\end{equation}
\begin{equation}\label{omega_smd_f2}
0\to P\to C_{1}\oplus C_{2}\to B \to 0.
\end{equation}
Combining (\ref{omega_smd_f2}) with the assumption on $\X$, we see that $C_{1}$, $C_{2}\in\X$.
Moreover, applying Lemma \ref{lprojext} to (\ref{omega_smd_f2}), we obtain that $C_{1}$, $C_{2}\in\mod_{I}(R)$.
Hence, by (\ref{omega_smd_f1}), it follows that $A_{1}$, $A_{2}\in\Omega(\X\cap\mod_{I}(R))$.
\end{proof}
%%%%%%%%%%%%%%%%%%%%
We end this section by providing the following lemma.  
For a local ring $(R, \m, k)$, it follows from Lemma \ref{filt} that $\Omega^{t+1} k$ belongs to $\Omega(\c_{\m}(R))$, where $t = \depth R$.  
However, the syzygy degree of $k$ can be reduced from $t+1$ to $t$.
\begin{lem}\label{t+1_trf}
Let $(R,\m,k)$ be a noetherian local ring, and set $t=\depth R$.
Then the module $\Omega^{t}k$ belongs to $\Omega(\c_{\m}(R))$.
\end{lem}
\begin{proof}
By \cite[Theorem 4.1(2)]{DeyTakahashi2023}, the module $\Omega^{t}k$ is $(t+1)$-torsionfree.
Hence, by \cite[1.1.1]{Iyama}, it is stably isomorphic to
\begin{equation}\label{omega_t_k}
\Omega^{t+1}\Tr\Omega^{t+1}\Tr\left(\Omega^{t}k\right)=\Omega\left(\Omega^{t}\Tr\Omega^{t+1}\Tr\Omega^{t}k\right).
\end{equation}
By Lemma \ref{filt} and Remark \ref{rmk_nfloc_stann}(3), the module $\Omega^{t}\Tr\Omega^{t+1}\Tr\Omega^{t}k$ belongs to $\c_{\m}(R)$.
It follows that the module in (\ref{omega_t_k}) belongs to $\Omega(\c_{\m}(R))$.
Combining this with Lemma \ref{omega_smd}, the assertion follows.
\end{proof}
%%%%%%%%%%%%%%%%%%%%
%%%%%%%%%%%%%%%%%%%%
\section{Finite $\Omega(\operatorname{CM}_{\mathfrak{m}}(R))$-type}
In this section, we state and prove Theorem \ref{mainthm} and provide its applications.  
We begin with the following elementary lemma concerning faithfully flat extensions of the annihilators of $\Ext$ modules.
%%%%%%%%%%%%%%%%%%%%
\begin{lem}\label{lem_flat_ann}
Let $R\to S$ be a flat local homomorphism of local rings.
Then for all modules $M\in\mod R$ and integers $n\geq0$, the following equality holds:
\[
\ann_{R}\Ext_{R}^{n}(M,\Omega^{n}M)=\ann_{S}\Ext_{S}^{n}(M\otimes_{R}S,\Omega_{S}^{n}(M\otimes_{R}S))\cap R.
\]
In particular, we have $(\sa_{R}M)S\subseteq \sa_{S}(M\otimes_{R}S)$.
\end{lem}
\begin{proof}
Note that the isomorphism $\Omega_{S}^{n}(M\otimes_{R}S)\cong\Omega^{n}(M)\otimes_{R}S$ holds.
Since $S$ is faithfully flat over $R$, we have
\[
\ann_{R}\Ext_{R}^{n}(M,\Omega^{n}M)=\ann_{S}\left(\Ext_{R}^{n}(M,\Omega^{n}M)\otimes_{R}S\right)\cap R.
\]
Combining this with the isomorphism $\Ext_{R}^{n}(M,\Omega^{n}M)\otimes_{R}S\cong\Ext_{S}^{n}(M\otimes_{R}S,\Omega_{S}^{n}(M\otimes_{R}S))$, the assertion follows.
\end{proof}
%%%%%%%%%%%%%%%%%%%%
The following lemma deals with the zero-dimensional case after factoring out a regular sequence in the proof of Theorem \ref{mainthm}.
\begin{lem}\label{lem_syz}
Let $(R,\m,k)$ be a noetherian local ring, and $I$ an ideal of $R$ such that $\m^{n}\subseteq I\subseteq\m$ for some positive integer $n$.
Then for every module $M$ in $\mod R/I$, we have $\Omega_{R/I}M\in\mod R/\m^{n-1}=\left|k\right|_{n-1}^{R}$.
Therefore, one has 
\[
\Omega_{R/I}(\mod R/I)\subseteq\left|k\oplus R/I\right|_{n-1}^{R}.
\]
\end{lem}
\begin{proof}
Let $M$ be a finitely generated $R/I$-module.
Since the syzygy $\Omega_{R/I}M$ is taken from a minimal free resolution of $M$, it is contained in $(\m/I)(R/I)^{\oplus s}$ for some $s \geq 0$.
Hence, we have
\[
\m^{n-1} \cdot \Omega_{R/I}M \subseteq \m^{n-1} \cdot (\m/I)(R/I)^{\oplus s} = 0.
\]
\end{proof}
%%%%%%%%%%%%%%%%%%%%
We now state the main result of this section.
Let $(R,\m,k)$ be a noetherian local ring.
We denote by $\gll(R)$ the {\em generalized Loewy length} of $R$, defined as the infimum of the Loewy lengths of artinian rings obtained by factoring $R$ by a system of parameters.
That is,
\[
\gll(R)=\inf\{n\geq1\mid\text{there exists a system of parameters }\xx\text{ of }R \text{ such that }\m^{n}\subseteq(\xx)\}.
\]
Note that $\gll(R)=1$ if and only if $R$ is regular.
\begin{thm}\label{mainthm}
Let $(R,\m,k)$ be a singular Cohen--Macaulay local ring of dimension $d$ and $I$ an $\m$-primary ideal of $R$.
We set $r=\inf\{n\geq1\mid\text{there exists a system of parameters }\xx\text{ of }R\text{ such that }\m^{n}\subseteq(\xx)\subseteq I\}$.
\begin{enumerate}[\rm(1)]
\item
One has $\Omega(\cm_{I}(R))\subseteq\left|\Omega^{d}k\oplus R\right|_{r-1}^{R}$.
\item
We have $\Omega(\cm_{\m}(R))\subseteq\left|\Omega^{d}k\oplus R\right|_{\gll(R)-1}^{R}$.
Moreover, if $\gll(R)=2$ or $R$ has minimal multiplicity, then the equality $\Omega(\cm_{\m}(R))=\add(\Omega^{d}k\oplus R)$ holds.
\item
Assume that $R$ is Gorenstein.
Then the following hold.
\begin{enumerate}[\rm(a)]
\item
One has $\cm_{I}(R)\subseteq\left|\Omega^{d}k\oplus R\right|_{r-1}^{R}$.
\item
We have $\cm_{\m}(R)\subseteq\left|\Omega^{d}k\oplus R\right|_{\gll(R)-1}^{R}$.
Moreover, if $\gll(R)=2$ or $R$ has minimal multiplicity, then the equality $\cm_{\m}(R)=\add(\Omega^{d}k\oplus R)$ holds.
\end{enumerate}
\end{enumerate}
\end{thm}
\begin{proof}
(1) Let $\bm{x}=x_{1},\ldots,x_{d}$ be a system of parameters of $R$ such that $\m^{r}\subseteq(\bm{x})\subseteq I$, and $M$ a module in $\cm_{I}(R)$.
Since $M$ is a maximal Cohen--Macaulay $R$-module, the sequence $\bm{x}$ is an $M$-sequence.
Moreover, since $\bm{x}\subseteq I\subseteq\sa_{R}(M)$, it follows from \cite[Corollary 3.2]{kos} that $M$ is a direct summand of $\Omega^{d}(M/\bm{x}M)$.
On the other hand, by Lemma \ref{lem_syz}, we have $\Omega_{R/(\bm{x})}(M/\bm{x}M)\in\left|k\oplus R/(\bm{x})\right|_{r-1}^{R}$.
By \cite[Lemma 4.2]{Nasseh Takahashi}, we obtain an isomorphism
\[
\Omega^{d+1}(M/\bm{x}M)\oplus R^{\oplus a}\cong\Omega^{d}\Omega_{R/(\bm{x})}(M/\bm{x}M)
\] 
for some $a\geq0$.
Combining the above, we conclude that $\Omega M\in\Omega^{d}\left(\left|k\oplus R/(\bm{x})\right|_{r-1}^{R}\right)\subseteq\left|\Omega^{d}k\oplus R\right|_{r-1}^{R}$.

(2) The first assertion follows from (1) by taking $I=\m$ in the definition of $r$.
Combining Lemmas \ref{filt}, \ref{omega_smd} and \ref{t+1_trf}, we see that the category $\Omega(\cm_{\m}(R))$ is closed under finite direct sums and direct summands, and contains $\Omega^{d}k$. 
This implies that 
\begin{equation}\label{mainthm_f1}
\Omega(\cm_{\m}(R))\supseteq\add(\Omega^{d}k\oplus R).
\end{equation}
We will prove the reverse inclusion in (\ref{mainthm_f1}) under each assumption.
Assume that $\gll(R)=2$. 
Then the inclusion 
\[
\Omega(\cm_{\m}(R))\subseteq\add(\Omega^{d}k\oplus R)
\]
follows from the previous result.
Now assume that $R$ has minimal multiplicity.
Let $S=R[x]_{\m R[x]}$ and consider the faithfully flat extension $R\to S$.
Note that $S=(S,\m S)$ is a $d$-dimensional Cohen--Macaulay local ring with minimal multiplicity and infinite residue field.
Hence, $\gll(S)=2$ by \cite[Exercise 4.6.14(c)]{BH}.
Let us consider the following subcategory of $\mod S$:
\[
\X=\{M\otimes_{R}S\mid M\in\cm_{\m}(R)\}.
\]
By Lemma \ref{lem_flat_ann}, we have $\X\subseteq\cm_{\m S}(S)$.
Then, applying the previous result to $S$, we obtain
\[
\Omega_{S}(\cm_{\m S}(S))\subseteq\add_{S}\left(\Omega_{S}^{d}(S/\m S)\oplus S\right)=\add_{S}\left(\left(\Omega^{d}k\oplus R\right)\otimes_{R}S\right).
\]
Note that $\{A\otimes_{R}S\mid A\in\Omega(\cm_{\m}(R))\}=\Omega_{S}\X\subseteq\Omega_{S}(\cm_{\m S}(S))$.
Hence, 
\[
\{A\otimes_{R}S\mid A\in\Omega(\cm_{\m}(R))\}\subseteq\add_{S}\left(\left(\Omega^{d}k\oplus R\right)\otimes_{R}S\right).
\]
Combining this with \cite[Proposition 2.18]{LW}, we conclude that
\[
\Omega(\cm_{\m}(R))\subseteq\add_{R}(\Omega^{d}k\oplus R).
\]

(3) It suffices to show that $\Omega(\cm_{I}(R))=\cm_{I}(R)$ for all ideals $I$ of $R$.
Let $M$ be a module in $\cm_{I}(R)$.
Since $R$ is Gorenstein, the equality $\cm(R)=\Omega\cm(R)$ holds.
Hence, there exists $N\in\cm(R)$ such that $M=\Omega N\oplus R^{\oplus a}$ for some $a\geq0$.
It follows that $\lEnd_{R}(M)=\lEnd_{R}(\Omega N)\cong\lEnd_{R}(N)$.
Therefore, we have $I\subseteq\sa_{R}(M)=\sa_{R}(N)$, and thus $N\in\cm_{I}(R)$.
This shows that $M\in\Omega(\cm_{I}(R))$.
\end{proof}
%%%%%%%%%%%%%%%%%%%%
\begin{rmk}
Let $(R,\m,k)$ be a Cohen--Macaulay local ring of dimension $d$.
\begin{enumerate}[\rm(1)]
\item
Note that if $k$ is infinite and $R$ has minimal multiplicity, then one has $\gll(R)\leq2$ by \cite[Exercise 4.6.14(c)]{BH}.
This fact was also used in the proof of Theorem \ref{mainthm}(2).
\item
Let $I$ be an $\m$-primary ideal of $R$, and let $\bm{x}=x_{1},\ldots,x_{d}$ be a system of parameters of $R$ satisfying $\m^{r}\subseteq(\bm{x})\subseteq I$.
For an integer $n>0$, set $l=d(n-1)+1$.
Then
\[
\m^{rl}\subseteq(\bm{x})^{l}\subseteq(\bm{x}^{n})\subseteq I^{n}.
\]
Hence, by Theorem \ref{mainthm}(1), we obtain
\[
\Omega(\cm_{I^{n}}(R))\subseteq\left|\Omega^{d}k\oplus R\right|_{rl-1}^{R}.
\]
In particular, taking $I=\m$ yields
\begin{equation}\label{mainthm_rmk_f1}
\Omega(\cm_{\m^{n}}(R))\subseteq\left|\Omega^{d}k\oplus R\right|_{\gll(R)(d(n-1)+1)-1}^{R}.
\end{equation}
On the other hand, by taking the union over $n$ in (\ref{mainthm_rmk_f1}), we obtain
\[
\Omega(\cm_{\V(\m)}(R))\subseteq\ext(\Omega^{d}k\oplus R).
\]
However, by \cite[Corollary 2.6]{stcm}, the equality $\cm_{\V(\m)}(R)=\ext(\Omega^{d}k\oplus R)$ holds.
Hence, the inclusion above turns out to be a trivial consequence.
\item
By the proof of Theorem \ref{mainthm}(1), one has
\[
\cm_{I}(R)\subseteq\Omega^{d}(\mod R/(\bm{x}))\subseteq\Omega^{d}(\mod R/\m^{r})\subseteq\left|\Omega^{d}k\oplus R\right|_{r}^{R}.
\]
The important point here is that taking a syzygy reduces the generation time by one.
\end{enumerate}
\end{rmk}
%%%%%%%%%%%%%%%%%%%%
Recall that a subcategory $\X$ of $\mod R$ is {\em of finite type} if the number of isomorphism classes of indecomposable modules in $\X$ is finite.
Note that if $R$ is local and $\X$ is closed under finite direct sums and direct summands, then $\X$ is of finite type if and only if $\X=\add M$ for some $M\in\X$ (see \cite[Theorem 2.2]{LW} for instance).
The following result is an application of Theorem \ref{mainthm}, and it provides a sufficient condition for the syzygy category of maximal Cohen--Macaulay modules to be of finite type.
\begin{cor}\label{cor_fin_ocm}
Let $(R,\m,k)$ be a singular Cohen--Macaulay local ring of dimension $d$.
Assume that $\m\subseteq\ca^{d+1}(R)$, and that either $\gll(R)=2$ or $R$ has minimal multiplicity.
Then one has $\operatorname{\Omega CM}(R)=\add(\Omega^{d}k\oplus R)$.
In particular, the category $\operatorname{\Omega CM}(R)$ is of finite type.
\end{cor}
%%%%%%%%%%%%%%%%%%%%
We give an example as an application of Corollary \ref{cor_fin_ocm}.
\begin{eg}\label{numring}
Let $e\ge2$, and consider the numerical semigroup ring 
\[
R=k[\![t^{e},t^{e+1},\ldots,t^{2e-1}]\!].
\]
This ring has also been studied in \cite[Example 2.11]{Kobayashi2017} and \cite[Example 3.4]{Esenetepe}.
Note that $R$ is a one-dimensional complete noetherian local domain with minimal multiplicity, and its conductor $C_{R}$ coincides with the maximal ideal $\m$ of $R$.
By \cite[Corollary 3.2]{Wang}, $C_{R}\subseteq\ca^{2}(R)$.
Hence, by Corollary \ref{cor_fin_ocm}, the category $\operatorname{\Omega CM}(R)$ is of finite type.
Finally, note that $R$ is Gorenstein if and only if $e\leq2$.
\end{eg}
%%%%%%%%%%%%%%%%%%%%
We next provide an example in which $\Omega(\cm_{\m}(R))$ is of finite type, while $\operatorname{\Omega CM}(R)$ is not.
Before proceeding, we remark that if $\operatorname{\Omega CM}(R)$ is of finite type, then $R$ must have an isolated singularity.
This is established in \cite[Proposition 4.10]{DKLO} and also follows from \cite[Theorem 3.8]{M3} and \cite[Theorem 2.2]{LW}.
\begin{eg}\label{formalpowerext}
Let $(S,\n)$ be a singular Cohen--Macaulay local ring with minimal multiplicity.
Consider the formal power series ring $R=S[\![x]\!]$.
Then $R=(R,\m)$ is again a singular Cohen--Macaulay local ring with minimal multiplicity.
Hence, by Theorem \ref{mainthm}(2), the category $\Omega(\cm_{\m}(R))$ is of finite type.
However, since $R$ does not have an isolated singularity, the category $\operatorname{\Omega CM}(R)$ is not of finite type.
\end{eg}
%%%%%%%%%%%%%%%%%%%%
The following corollary states that if the cohomology annihilator of a Cohen--Macaulay local ring $R$ coincides with its maximal ideal, then the module $\Omega^{d}k\oplus R$ serves as a strong generator for $\mod R$ with generation time at most $\gll(R)-1$ (cf. \cite[Theorem~1.1]{DKLO}). For the notion of a strong generator in the module category, see \cite[4.3]{Iyengar Takahashi 2016}.
\begin{cor}\label{fin_syz_type}
Let $(R,\m,k)$ be a Cohen--Macaulay local ring of dimension $d$.
Assume that $\ca(R)=\ca^{s+1}(R)=\m$ for some $s\geq0$.
Then the following statements hold.
\begin{enumerate}[\rm(1)]
\item
One has $s\geq d$, and $\Omega^{s}(\mod R)\subseteq\cm_{\m}(R)$.
\item
One has $\Omega^{s+1}(\mod R)\subseteq\left|\Omega^{d}k\oplus R\right|_{\gll(R)-1}^{R}$.
\item
If $\gll(R)=2$ or $R$ has minimal multiplicity, then we have $\Omega^{s+1}(\mod R)\subseteq\add(\Omega^{d}k\oplus R)$.
\end{enumerate}
\end{cor}
\begin{proof}
(1) Since $\ca(R)=\m\neq R$, the ring $R$ is singular.
In particular, $\m\neq0$, and hence $\ca(R)\neq0$.
By \cite[Proposition 2.4]{Iyengar Takahashi 2016}, we have $\ca^{l}(R)=0$ for all $l\leq d$.
Therefore, $s\geq d$.
It follows that
\[
\Omega^{s}(\mod R)\subseteq\Omega^{d}(\mod R)\subseteq\cm(R).
\]
In addition, since $\m=\ca^{s+1}(R)=\sa_{R}(\Omega^{s}(\mod R))$, we have $\Omega^{s}(\mod R)\subseteq\mod_{\m}(R)$.
Thus, we obtain that
\[
\Omega^{s}(\mod R)\subseteq\cm(R)\cap\mod_{\m}(R)=\cm_{\m}(R).
\]

The remaining two assertions follow from (1) together with Theorem \ref{mainthm}(2).
\end{proof}
%%%%%%%%%%%%%%%%%%%%
\begin{rmk}
Let $(R, \m, k)$ be a noetherian local ring.
Assume that $\ca(R)$ is an $\m$-primary ideal.
Then there exist integers $n > 0$ and $s \geq \dim R$ such that
\[
\m^{n} \subseteq \ca(R) = \ca^{s+1}(R).
\]
Hence, we have
\[
\Omega^{s}(\mod R) \subseteq \c_{\m^{n}}(R).
\]
\end{rmk}
%%%%%%%%%%%%%%%%%%%%
In the above corollary, when $R$ is Gorenstein, we recover \cite[Theorem 5.5]{DeyLiu}, which asserts that a Gorenstein local ring with minimal multiplicity and with cohomology annihilator equal to the maximal ideal of $R$ is of finite Cohen--Macaulay representation type.
\begin{cor}\label{DeyLiu5.6}
Let $(R,\m,k)$ be a Gorenstein local ring of dimension $d$.
Assume that $\ca(R)=\m$.
Then the following hold.
\begin{enumerate}[\rm(1)]
\item
The equality $\cm(R)=\left|\Omega^{d}k\oplus R\right|_{\gll(R)-1}^{R}$ holds.
\item
Assume that either $\gll(R)=2$ or $R$ has minimal multiplicity. 
Then the equality $\cm(R)=\add(\Omega^{d}k\oplus R)$ holds.
\end{enumerate}
\end{cor}
\begin{proof}
Since $R$ is Gorenstein, it follows from \cite[Proposition 3.4]{DeyLiu} that $\ca(R)=\ca^{d+1}(R)$.
In addition, we have $\Omega^{d+1}(\mod R)=\cm(R)$.
Therefore, the assertions follow from Corollary \ref{fin_syz_type}.
\end{proof}
%%%%%%%%%%%%%%%%%%%%
%%%%%%%%%%%%%%%%%%%%
\begin{ac}
The author would like to thank his supervisor Ryo Takahashi for giving many thoughtful questions and helpful discussions. The author also thanks Yuya Otake and Kaito Kimura for giving valuable comments, and Jian Liu for sharing the paper \cite{DeyLiu}.
\end{ac}
%%%%%%%%%%%%%%%%%%%%

\end{document}